\documentclass[10pt]{amsart}

\usepackage{amssymb,graphicx, amsmath, amsthm, MnSymbol, array}
\usepackage{calrsfs}
\usepackage{wrapfig}
\usepackage{graphicx}
\graphicspath{ {./images/} }

\usepackage{todonotes}

\usepackage{floatflt}

\usepackage{hyperref}

\def\m{{\mathfrak m}} 

\def\c{\mathcal}

\def\m{{\mu}}

\begin{document}

\newtheorem{theorem}{Theorem}[section]
\newtheorem{lemma}[theorem]{Lemma}
\newtheorem{proposition}[theorem]{Proposition}
\newtheorem{corollary}[theorem]{Corollary}
\newtheorem{problem}[theorem]{Problem}
\newtheorem{construction}[theorem]{Construction}

\theoremstyle{definition}
\newtheorem{defi}[theorem]{Definitions}
\newtheorem{definition}[theorem]{Definition}
\newtheorem{notation}[theorem]{Notation}
\newtheorem{remark}[theorem]{Remark}
\newtheorem{example}[theorem]{Example}
\newtheorem{question}[theorem]{Question}
\newtheorem{comment}[theorem]{Comment}
\newtheorem{comments}[theorem]{Comments}

\newtheorem{discussion}[theorem]{Discussion}

\renewcommand{\thedefi}{}

\long\def\alert#1{\smallskip{\hskip\parindent\vrule%
\vbox{\advance\hsize-2\parindent\hrule\smallskip\parindent.4\parindent%
\narrower\noindent#1\smallskip\hrule}\vrule\hfill}\smallskip}

\def\ff{\frak}
\def\tf{torsion-free}
\def\Spec{\mbox{\rm Spec }}
\def\Proj{\mbox{\rm Proj }}
\def\hgt{\mbox{\rm ht }}
\def\type{\mbox{ type}}
\def\Hom{\mbox{ Hom}}
\def\rank{\mbox{ rank}}
\def\Ext{\mbox{ Ext}}
\def\Tor{\mbox{ Tor}}
\def\Ker{\mbox{ Ker }}
\def\Max{\mbox{\rm Max}}
\def\End{\mbox{\rm End}}
\def\xpd{\mbox{\rm xpd}}
\def\Ass{\mbox{\rm Ass}}
\def\emdim{\mbox{\rm emdim}}
\def\epd{\mbox{\rm epd}}
\def\repd{\mbox{\rm rpd}}
\def\ord{\mbox{\rm ord}}

\def\DD{{\mathcal D}}
\def\EE{{\mathcal E}}
\def\FF{{\mathcal F}}
\def\GG{{\mathcal G}}
\def\HH{{\mathcal H}}
\def\II{{\mathcal I}}
\def\LL{{\mathcal L}}
\def\MM{{\mathcal M}}
\def\PP{{\mathcal P}}

\title[The conic geometry of rectangles inscribed in lines]{The conic geometry of rectangles \\ inscribed in lines}

\author{Bruce Olberding} 
\address{Department of Mathematical Sciences, New Mexico State University, Las Cruces, NM 88003-8001}

\email{bruce@nmsu.edu}

\author{Elaine A.~Walker}
\address{1801 Imperial Ridge, Las Cruces, NM 88011}

\email{miselaineeous@yahoo.com}

\begin{abstract} We develop a circle of ideas involving pairs of lines in the plane, intersections of hyperbolically rotated elliptical cones and the locus of the centers of  rectangles inscribed in lines in the plane. \end{abstract}


\thanks{\today}

\maketitle

\section{Introduction}

The problem that motivates this article   is that of finding all the rectangles inscribed in a  set of lines in the plane, i.e., all the rectangles whose vertices lie on lines in this set. Since any such rectangle has vertices on at most four of the lines,  this  reduces  to  finding all the  rectangles inscribed in   four  lines, a problem that 
 can   be   further broken down into that of finding all rectangles inscribed in {\it two pairs of lines such that one diagonal of the rectangle joins the lines in one pair and the other diagonal joins the lines in the other pair}. 
 By  varying the pairings of the four lines it is possible to catalog all the inscribed rectangles based on the sequence of lines on which the vertices lie in clockwise fashion.  
   (For example, if the lines are $A,B,C,D$, then the pairings $A,C$ and $B,D$ encode the rectangles whose vertices lie in the sequence $ABCD$ or $DCBA$.) 
{This reduction, a variation of which is called the ``permutation trick'' by Schwartz in \cite[Section 4.1]{Sch2}, allows us to focus on two pairs of lines.}  
A direct equational approach to describing the rectangles inscribed in the two pairs of lines results in cumbersome and rather opaque equations,
so while we do give a solution to this problem, our real interest lies in developing a geometric way to recast the problem in terms of elliptical cones and conic sections in order that these and related questions  can be dealt with by other than computational or  {\it ad hoc} means.

{Our approach is different from that of Schwartz in the recent article \cite{Sch}, which  deals also with rectangles inscribed in lines. Whereas we use  techniques involving linear algebra and conic geometry to make a theoretical framework for the problem,  
 Schwartz makes connections to  hyperbolic geometry and applies 
 computational and geometric methods to identify finer  features of the space of inscribed rectangles. 
 Despite the difference in approaches, there   
is interesting overlap at a few points. For example, 
 Schwartz proves that 
for a  ``nice'' configuration of four lines, the locus of points consisting of the centers of rectangles whose vertices lie in sequence on these four lines is a hyperbola. We prove  the same thing in Theorem~\ref{indiff}  with a different technique but in more generality and with a converse. (The cases in which the locus is not a hyperbola are dealt with in Proposition 4.4.\footnote{In \cite{OW2}, using another approach to the different cases,  we show that while the locus need not be a hyperbola, it is the image of a hyperbola under an affine transformation, a hyperbola that encodes the $x$-coordinates of the rectangle vertices on $A$ and $B$.})
 So although our article is self-contained, it is useful to compare it to Schwartz's article \cite{Sch} because the two approaches have  complementary methods that lead in different directions.

Graphs and animations illustrating several of the geometrical  ideas in the paper can be found in \cite{Zia}.

\section{Hyperbolically rotated cones}

What we call hyperbolically rotated cones belong to the class of real elliptical cones in ${\mathbb{R}}^3$ that have  apex in the $xy$-plane and central 
axis parallel to the $z$-axis. For such a  cone $\c A$ there is a positive definite symmetric $2\times 2$ matrix~$A$ and a point ${\bf a}$ in the $xy$-plane such that the defining equation for $\c A$ in variables ${\bf x} = (x,y)$ and $z$  is \begin{eqnarray} \label{form}
{z^2} &  = &  
 ({\bf x}-{\bf a})^T A ({\bf x}-\bf{a}).
 \end{eqnarray} 
 The point ${\bf a}$ is where the apex of the cone resides in the $xy$-plane, and it is the center of the ellipses that are the level curves of the cone. 
 We say that $A$ is the {\it cone matrix} for the cone $\c A$.   To emphasize the dependence of the cone ${\mathcal{A}}$ on the matrix  $A$ and point ${\bf a}$, we 
write ${\mathcal{A}} = (A,{\bf a})$.  (We treat points as vectors throughout.)

\begin{proposition} 
 \label{ugh} If $(A,{\bf a})$ and $(B,{\bf b})$ represent the same elliptical cone, then $A = B$ and ${\bf a} = {\bf b}$. 
 \end{proposition}
 
 \begin{proof}
Let $\alpha =  {\bf a}^TA{\bf a}+1$, and  
consider the  polynomials $f,g \in {\mathbb{R}}[x,y]$ of degree $2$ defined for ${\bf x}=(x,y)$ by 
$$f({\bf x}) = ({\bf x}-{\bf a})^TA({\bf x}-{\bf a}) -\alpha,  \:
g({\bf x}) = ({\bf x}-{\bf b})^TB({\bf x}-{\bf b}) -\alpha.$$
Since $A$ is positive definite, $\alpha \geq 1$. 
Each of these polynomials defines the level curve of the cone at height $\sqrt{\alpha}$ and hence both $f$ and $g$ 
define the same ellipse. Since $f$ and $g$ are irreducible and share infinitely many zeroes, there is a real number $\gamma$ such that 
%
  %
   %
   $f = \gamma g$ and hence $A = \gamma B$.  
 Moreover, since $({\bf 0} -{\bf a})^T A ({\bf 0} -{\bf a})= {\bf a}^TA{\bf a}\geq 0$, $({\bf 0} -{\bf b})^T B ({\bf 0} -{\bf b})= {\bf b}^TB{\bf b} \geq 0$ and $(A,{\bf a})$ and $(B,{\bf b})$ represent the same cone,    it follows that ${\bf a}^TA{\bf a} = {\bf b}^TB{\bf b}.$ 
 Therefore,  ${\bf a}^TA{\bf a}-\alpha = f({\bf 0})  = \gamma g({\bf 0})  =\gamma ( {\bf a}^TA{\bf a}-\alpha)$, so  $\gamma  = 1$  
  since $\alpha \ne {\bf a}^TA{\bf a}$. Thus $A = B$. 
Also, since  $(A,{\bf a})$ and $(B,{\bf b})$ represent the same cone, $0=({\bf a}-{\bf a})^T A ({\bf a}-{\bf a}) = ({\bf a}-{\bf b})^T B ({\bf a}-{\bf b})$. 
  As $B$ is positive definite, this implies  ${\bf a} = {\bf b}$. 
  \end{proof}

Among the elliptical cones defined by Equation~(\ref{form}), we single out those  that in the next section will serve as our three-dimensional expression of a pair of lines in the plane. 
We explain the reason for the terminology below.

 \begin{definition} A {\it hyperbolically rotated cone} ({\it HR-cone} for short) is 
   a  real elliptical cone as in Equation~(\ref{form}) whose cone matrix has determinant $1$.    
 \end{definition}

Since an HR-cone is specified by a cone matrix and a position in the plane, the HR-cones having apex at a given location  in the plane are parameterized by the group  $SL(2,{\mathbb{R}})$.

  \begin{notation} \label{nota} We denote by $I$ the $2 \times 2$ identity matrix. 
 For  real numbers $\phi$ and $0 < \theta < \frac{\pi}{2}$, 
   define 
 $$R_\phi = \begin{bmatrix} \cos \phi & -\sin \phi \\
\sin \phi & \cos \phi \end{bmatrix}, \: \Lambda_\theta =  \begin{bmatrix} \tan^2\theta &0  \\
0& \cot^2\theta \end{bmatrix}.$$ The matrix $R_\phi$
is  the (Euclidean) rotation matrix  that rotates a vector in the plane by $\phi$ radians in the counterclockwise direction, while  $\Lambda_\theta$ is a  squeeze matrix. 
\end{notation} 

Since $\tan^2(\theta)\cot^2(\theta) =1$, the matrix  $\Lambda_\theta$ 
  represents a  hyperbolic rotation in the sense that any point on the hyperbola $y = 1/x$ is ``rotated'' to another point on this same hyperbola  by  the matrix. 
 This  motivates the  terminology for our cones: Let  ${\mathcal{A}}=(A,{\bf a})$ be a cone as in Equation~(\ref{form}).  
Since $A$ is symmetric and positive definite, the square root $A^{\frac{1}{2}}$ of $A$ exists. 
If also $\c A$ is an HR-cone, 
the spectral theorem and the fact that $\det(A) = 1$ imply  
%
there are real numbers $\phi$ and $0<\theta< \frac{\pi}{2}$ with $$A^{\frac{1}{2}} = R_{\phi}(\Lambda_\theta)^{\frac{1}{2}}R_{-\phi}.$$
%
 The  geometric interpretation of this decomposition is that $A^{\frac{1}{2}}$ 
is a hyperbolic rotation along the axes obtained by rotating the $x$- and $y$-axes by $\phi$ radians. 
This hyperbolic rotation  accounts  for the shape of the  cone ${\mathcal{A}}$: Since the cone ${\mathcal{A}}$ is defined by 
$z^2 = ({\bf x}-{\bf a})^T A ({\bf x}-{\bf a})$ and  we have  $$({\bf x}-{\bf a})^T A ({\bf x}-{\bf a}) = 
(A^{\frac{1}{2}}({\bf x}-{\bf a}))^T(A^{\frac{1}{2}} ({\bf x}-{\bf a})),$$ the cone
${\mathcal{A}}$ is the image under the linear transformation induced by $A^{-\frac{1}{2}}$ of what we call a unit cone.

\begin{definition} 
A {\it unit cone} is a circular HR-cone. Thus a unit cone is 
an HR-cone of the form $(I,{\bf a})$ and its  
defining equation   is    $z^2 = ({\bf x}-{\bf a})^T ({\bf x}-{\bf a})$.
\end{definition}

%
In summary,  the cone   ${\mathcal{A}} = (A,{\bf a})$  is the image of the unit cone  under a linear transformation that  hyperbolically rotates the $x,y$ coordinates by $A^{-\frac{1}{2}}$ and leaves the $z$-coordinate fixed.  Thus a hyperbolically rotated cone is more precisely  a {\it hyperbolically rotated unit cone}.

This  leads to a simple criterion for distinguishing HR-cones. 
For a cone ${\c A} = (A,{\bf a})$ as in Equation~(\ref{form}), the  linear transformation induced by the  matrix $A^{-\frac{1}{2}}$  expands area by a factor of $(\det(A))^{-\frac{1}{2}}$. Since ${\mathcal{A}}$ 
 is the image of the unit cone under the linear transformation induced by $A^{-\frac{1}{2}}$, it follows that $\det(A) = 1$ if and only if the level curve $\c E$ at height~$1$ has area $\pi$. Therefore: 
{\it A real elliptical cone ${\c A}=(A,{\bf a}) $ as in Equation~$(\ref{form})$ is an HR-cone if and only if  the level curve   at height $1$   has area $\pi$.}

A similar argument shows that the level curve of an HR-cone at height $h$  has  area $\pi h^2$ and so the product of the lengths of the semi-major and minor axes of this ellipse is $h^2$.  
The only thing special about 
 the ellipses themselves that appear as level curves of HR-cones is the relationship between the  area of the level curve  and the height at which it occurs. Every ellipse $\c E$ in the plane occurs as a level curve of some HR-cone.  To see this, let $M$ and $m$ be the lengths of the semi-major and semi-minor axes, respectively. Then 
 the cone given by the equation $z^2   = mM^{-1}{x^2}+ m^{-1}M{y^2}$ has a level curve at height $z= \sqrt{Mm}$ that can be transformed into $\c E$ by rotation and translation.

As this suggests, any  real elliptical cone whose equation is given by  Equation~(\ref{form}) can be rescaled to obtain an HR-cone.

\section{The surface defined by a pair of lines}

We now consider pairs of   lines $L_1,L_2$ in the plane. (Throughout the paper, by a pair of lines we mean a pair of distinct lines.)
 We are specifically interested in the surface obtained by associating to each point $p$ in the plane half
 the length of the line segment through $p$ joining $L_1$ and $L_2$ and having $p$ as its midpoint.\footnote{Our definition of this surface is loosely motivated by Vaughn's proof (see  \cite[p.~71]{Mey}) that every Jordan closed curve has a rectangle inscribed in it.}  As we note in Lemma~\ref{basic}, as long as $L_1$ and $L_2$ are not parallel, 
 there is a unique such line segment.
  On the other hand, if $L_1$ and $L_2$  are parallel, then the set of points that occur as a midpoint of a line segment having endpoints on $L_1$ and $L_2$ is simply the set of points that are equidistant from $L_1$ and $L_2$.  If $p$ is a point on this line,  {\it every} line segment joining $L_1$ and $L_2$ and passing through $p$ has $p$ as its midpoint. 




\begin{definition} Let $L_1,L_2$  be a pair of    lines in the plane, and let ${\c S}$ 
be  the subset of ${\mathbb{R}}^3$ consisting of all  the points $(x,y,\pm z)$, where $2z$ is the  length of a line segment between $L_1$ and $L_2$ having $(x,y)$ as its midpoint. We say that $L_1$ and $L_2$ {\it define the surface ${\c S}$}. 
\end{definition}

The case where $L_1$ and $L_2$ are not parallel is the more intricate one, and we deal with it first. The case of parallel lines can be dispatched with quickly, and we do this at the end of the section.

 If $L_1$ and $L_2$ are intersecting and distinct lines, then 
  every point in the plane occurs as a midpoint of a {unique} line segment joining  $L_1$ and $L_2$.  This situation is summarized in the following lemma. To simplify notation, we assume that neither   $L_1$ nor $L_2$ is parallel to the $y$-axis, a situation that can always be obtained by a suitable rotation.

\begin{lemma} \label{basic}
Let $L_1$ and $L_2$ be distinct intersecting lines in the plane, neither of which is parallel to the $y$-axis, and let $m_i$ and $b_i$ denote the slope and $y$-intercept of $L_i$. 
For each point $(x,y)$   in the plane, there is  a unique line segment $L$ that has endpoints on $L_1$ and $L_2$  and midpoint $( x , y )$. The coordinates $(x_1,y_1)$ and $(x_2,y_2)$  of the endpoints on $L_1$ and $L_2$, respectively, are given by
\smallskip
 $$\begin{array}{ll}
\displaystyle{x_1
= \frac{-2m_2 x  +2 y -b_1-b_2}{m_1-m_2}} &
 \displaystyle{y_1 = \frac{-2m_1m_2 x +2m_1 y -b_1m_2-b_2m_1}{m_1-m_2}} \\ \\
\displaystyle{x_2 = \frac{2m_1 x -2y +b_1+b_2}{m_1-m_2}} & 
\displaystyle{y_2=\frac{2m_1m_2 x -2m_2 y+b_1m_2+b_2m_1 }{m_1-m_2}}.
\end{array}$$
\end{lemma}

\begin{proof}  The proof is  a matter of routine calculation. 
Let $( x , y )$ be a point in the plane,  let $x_1,x_2 \in {\mathbb{R}}$, and let $y_1 = m_1x_1+b_1,y_2 = m_2x_2+b_2$. Then $(x_1,y_1) \in L_1$ and $(x_2,y_2) \in L_2$.  Now $( x , y )$ is the midpoint of the line segment joining $(x_1,y_1) $ and $(x_2,y_2) $ if and only if $2 x  = x_1+x_2$ and $2 y =y_1+y_2$. Since $y_1 + y_2 -b_1-b_2 = m_1x_1+m_2x_2$, it follows that $( x , y )$ is the midpoint of the line segment joining $(x_1,y_1) $ and $(x_2,y_2) $ if and only if  
$ 2 x   =  x_1 + x_2$  and $ 
2 y  - b_1-b_2  =  m_1x_1 + m_2x_2.$ 
Since $m_1 \ne m_2$, there are unique $x_1,x_2$ that satisfy these equations. Now solve for $x_1$ and $x_2$. 
\end{proof}



\begin{theorem} \label{cone 1} If  $L_1$ and $L_2$  are a pair of  distinct    lines in the plane that meet in a point ${\bf a}$, then the surface defined by the two lines is the  HR-cone    ${\c S}=(R_{\phi} \Lambda_\theta R_{-\phi},{\bf a}),$
  where $2\theta$ 
 is the measure of the smaller angle between the two lines and $\phi$ is the angle between the  $x$-axis and the line that bisects the smaller angle between the two lines.

 \end{theorem} 

\begin{proof}
We first prove the theorem in the case that ${\bf a}={\bf 0}$ and the $x$-axis bisects the smallest angle between $L_1$ and $L_2$. Let $z>0$, and let ${\mathcal{E}}$ be the set of all points $p$ in the plane such that $p$ is the midpoint of a line segment of length $2z$ with endpoints on $L_1$ and $L_2$. We show that ${\mathcal{E}}$ is an ellipse by finding its defining equation. 

The equations of $L_1$ and $L_2$ are given by $y = (\tan \theta)x $ and $y = -(\tan \theta)x$, respectively.  Let $(x,y) \in {\mathcal{E}}$, and let $(x_1,y_1),(x_2,y_2)$ be the endpoints on $L_1,L_2$ of the line segment through $(x,y)$ that has $(x,y)$ as its midpoint. Then 
$2x = x_1+x_2$  and  $2 y = y_1+y_2 = (\tan \theta)x_1 - (\tan \theta)x_2$. 
Since $0< \theta\leq\pi/4$, we have $\tan(\theta) \ne 0$, and so the second equation implies that $x_1 - x_2 = 2\cot(\theta)y.$  
These observations  together imply
\begin{eqnarray*}(2z)^2 & = & {(x_1-x_2)^2 + (y_1-y_2)^2} 
\:  = \: (2\cot(\theta)y)^2 + (\tan(\theta)x_1 + \tan(\theta)x_2)^2 \\
\: & = & 4(\cot^2 \theta)y^2 + 4(\tan^2 \theta)x^2.
\end{eqnarray*}
For each choice of $z>0$,  ${\mathcal{E}}$ is the ellipse in the plane whose defining equation is   $z^2 =(\tan^2 \theta)x^2+(\cot^2 \theta)y^2.$ Consequently, the surface defined by the lines $L_1$ and $L_2$ is 
 the HR-cone $(\Lambda_\theta,(0,0)).$

To see that the theorem holds in full generality, suppose that $L_1$, $L_2$, $\phi$ and $\theta$ are as in the statement of the  theorem. After a translation to the origin and a rotation of $\phi$ radians in the counterclockwise direction, the lines are in the position described in the first part of the proof, 
 and so  the surface defined by $L_1,L_2$ is the 
real elliptical cone with defining equation  $ 
z^2=({\bf x}-{\bf a})^T R_{\phi} \Lambda_\theta R_{-\phi} ({\bf x}-{\bf a}).$
This elliptical cone is hyperbolically rotated since $R_{\phi} \Lambda_\theta R_{-\phi}$ is symmetric, positive definite and has determinant $1$. 
\end{proof}

 \begin{remark} \label{axes rem}
Theorem~\ref{cone 1} implies that the major and minor axes of any level curve of the HR-cone ${\mathcal{C}}$ lie on the 
 pair of orthogonal  lines that bisect the defining lines for ${\mathcal{C}}$. 
 By the principal axis theorem, these lines are in the direction of any pair of linearly independent eigenvectors of the cone matrix for $\c C$. Similarly, the eigenvalues for the cone matrix are the squares of the lengths of the semi-major and semi-minor axes of the ellipse that is the level curve at height $z =1$. 
 \end{remark}

%
%


In Theorem~\ref{pos def} we show  how to obtain the defining lines
for an HR-cone
 from its cone matrix and apex location. This depends on the following lemma, which is needed  in Section 4 also. 

   \begin{lemma} \label{invertible} If $A \ne B$ are $2 \times 2$ symmetric positive definite matrices with determinant $1$, then $\det(A-B) <0$. 
   \end{lemma}
   
   \begin{proof} Since $A$ is symmetric, there is an orthogonal matrix $Q$ and a diagonal matrix $\Lambda$ such that $Q^T\Lambda Q = A$. 
  Because $\det(A) = 1$ there is $\alpha \ne 0$ with $$\Lambda=\begin{bmatrix} 
\alpha & 0 \\
0 & \frac{1}{\alpha}\\
\end{bmatrix}.$$
Also, since $A-B = Q^T(\Lambda-QBQ^T)Q$ and $\det(Q) = 1$, we have $\det(A-B) = \det(\Lambda-QBQ^T)$. 
We show that  $ \det(\Lambda-QBQ^T)<0$.

 Since $QBQ^T$ is symmetric, there are real numbers $a,b,d$ with $$QBQ^T = \begin{bmatrix} 
a & b \\
b & d\\
\end{bmatrix}.$$ 
Using the fact that $ad-b^2=\det(Q^TBQ) =1$, we have $$\det(A-B)=\det(\Lambda-QBQ^T) = (\alpha-a)\left(\frac{1}{\alpha}-d\right) -b^2 = 2-d\alpha-\frac{a}{\alpha}.$$

The fact that  $\alpha>0$ implies  $$2 -d\alpha -\frac{a}{\alpha} \leq 0 \: \Longleftrightarrow \: 2\alpha -d\alpha^2 -a\leq 0.$$  
Now  $d>0$ since $QBQ^T$ is positive definite, so the function $f(x)=2x -dx^2 -a$ attains its maximum value at $x=1/d$. Since $ad-b^2 = 1$,  the maximum value this function attains is $$f(d^{-1}) = \frac{2}{d} - \frac{1}{d} - a = \frac{1-ad}{d}=\frac{ -b^2}{d} \leq 0.$$  Therefore, $\det(A-B) \leq 0$. 

To rule out the case that  $\det(A-B) = 0$,  suppose otherwise. Then $f(\alpha)=0$, so that since $0=f(\alpha) \leq \frac{-b^2}{d}\leq 0$, 
 we have $b^2=0$,   $1 = ad-b^2 =ad$ and $a = d^{-1}$. 
So the assumption that $\det(A-B) =0$ implies  $$\Lambda-QBQ^T =  \begin{bmatrix} 
\alpha-a & 0 \\
0 & \frac{1}{\alpha} - \frac{1}{a}\\
\end{bmatrix}.$$ Therefore,  $$0 = \det(A-B)= (\alpha - a) \left(\frac{1}{\alpha}-\frac{1}{a}\right) =\frac{(\alpha-a)^2}{a\alpha}.$$
This implies  $\alpha = a$. From this, $a=d^{-1}$, and  $b=0$ we obtain $A =B$, a contradiction that implies 
 $\det(A-B) < 0$.
%
\end{proof}



\begin{theorem} \label{pos def}  

Let  ${\c S}$ be the surface
  defined by a pair of  distinct lines. 
   \begin{itemize}

\item[(1)]  If ${{\c S}}$ is an HR-cone but not a unit cone, then ${\c S}$ has a unique pair  of defining lines. With $\c S=(A,{\bf a})$, these lines  are given by the degenerate hyperbola $({\bf x}-{\bf a})^T(A-I)({\bf x}-{\bf a}) = 0.$

\item[(2)]  ${\c S}$   
    is a unit cone if and only if the defining lines for ${\c S}$ are orthogonal, which holds if and only if the cone matrix for ${\c S}$ is $I$. In this case, the lines in each pair of defining lines for 
   ${\c S}$ cross at the same point, and 
     every pair of orthogonal lines crossing at this point defines ${\c S}$.  

  \end{itemize}

\end{theorem}

\begin{proof} 
The spectral theorem implies that there are  real numbers $0\leq \alpha < \pi$  and $0< \beta \leq \frac{\pi}{4}$ such that $A = R_{\alpha}\Lambda_{\beta}R_{-\alpha}$. Therefore, Theorem~\ref{cone 1} implies that
$\c S$ is defined by the 
  two  lines in the plane  meeting at the point ${\bf a}$ such that  $2\beta$ 
 is the measure of the smaller angle between the two lines, and  $\alpha$ is the angle between   the $x$-axis and the line that bisects the smaller angle between the two lines. 
 
(1)
 Suppose that ${\mathcal{S}}$ is an HR-cone that is not the unit cone. Let $L_1$ and $L_2$ be defining lines for ${\mathcal{S}}$, and let ${\bf b}$ be the point where these lines intersect. 
 With  $2\theta$ 
the measure of the smaller angle between the two lines, and  $\phi$  the angle between the  $x$-axis and the line that bisects the smaller angle, we have 
by Theorem~\ref{cone 1} that $\c S = (R_{\phi}\Lambda_{\theta}R_{-\phi},{\bf b})$.  
By Proposition~\ref{ugh}, $A = R_{\phi}\Lambda_{\theta}R_{-\phi}$ and ${\bf b}={\bf a}$.  
After translation, we can assume that the point ${\bf a}$ is the origin.
 Since ${\mathcal{S}}$ is not the unit cone,  $A \ne I$, so that by Lemma~\ref{invertible}, $\det(A-I) <0$. Therefore, the equation ${\bf x}^T(A-I){\bf x} = 0$ defines a degenerate hyperbola (see for example \cite[p.~161]{Kendig}). To prove that this pair of lines is simply  
 $L_1$ and $L_2$, it suffices to show there are four points  on $L_1$ and $L_2$, not all collinear, that satisfy the equation
${\bf x}^TA{\bf x} = {\bf x}^T{\bf x}.$


Let
 \begin{center} 
 $\begin{bmatrix} 
u \\
v \\
\end{bmatrix}
=R_{\phi}\begin{bmatrix} 
\cos(\theta) \\
\sin(\theta) \\
\end{bmatrix}
=
\begin{bmatrix}
\cos(\phi+\theta) \\
\sin(\phi+\theta) \\
\end{bmatrix},
$ \:
 $\begin{bmatrix} 
s \\
t \\
\end{bmatrix}
=R_{\phi}\begin{bmatrix} 
\cos(-\theta) \\
\sin(-\theta) \\
\end{bmatrix}
=\begin{bmatrix} 
\cos(\phi-\theta) \\
\sin(\phi-\theta) \\
\end{bmatrix}$.
\end{center} 

  The angles between the $x$-axis and each of the lines $L_1$ and $L_2$ are $\phi+\theta$ and $\phi-\theta$, and so 
one of the lines $L_1$, $L_2$  goes through the points 
$(s, t)$ and $(-s,  -t)$ and the other  goes through $(u,v)$ and $(-u,-v)$.
 Since each of these points satisfies the equation ${\bf x}^T{\bf x}=1$, 
we need only show  they also  satisfy the equation
${\bf x}^TA{\bf x} = 1.$

 Since 
$A =  R_{\phi}\Lambda_{\theta}R_{-\phi}$,  we have 
\begin{eqnarray*} 
\begin{bmatrix} 
u &
v \\
\end{bmatrix}A\begin{bmatrix} 
u \\
v \\
\end{bmatrix}
& = & 
\begin{bmatrix} 
\cos(\theta) &
\sin(\theta) \\
\end{bmatrix}R_{-\phi}
R_{\phi} \begin{bmatrix} 
\tan^2 \theta & 0\\ 
0 & \cot^2 \theta  \\
\end{bmatrix}
R_{-\phi}R_{\phi}
\begin{bmatrix} 
\cos(\theta) \\
\sin(\theta) \\
\end{bmatrix}
\\
& = & \cos^2(\theta) \tan^2(\theta) + \sin^2(\theta)\cot^2(\theta) \: = \: 1.
\end{eqnarray*}
Thus  $(u,v)$ and $(-u,-v)$ satisfy the equation  ${\bf x}^TA{\bf x} = 1.$ A similar calculation shows that that $(s,t)$ and $(-s,-t)$ also satisfy this equation, which verifies (1).

(2) By Proposition~\ref{ugh}, $(A,{\bf a})$ is a unit cone if and only if $A=I$.  
Moreover, 
if a pair of  defining lines for ${\mathcal{S}}$ is not an orthogonal pair, then Theorem~\ref{cone 1} implies that $A \ne I$ and hence ${\mathcal{S}}$ is not  a unit cone. Conversely, if a pair of defining lines is orthogonal, then Theorem~\ref{cone 1} implies that  $A = I$, so that ${\mathcal{S}}$ is the unit cone. 

Now suppose that $\c S = (I,{\bf a})$ is a unit cone. 
To see that every pair of orthogonal lines through ${\bf a}$ defines $\c S$, 
  let $L_1,L_2$ be a pair of orthogonal lines through ${\bf a}$. Then $\theta = \frac{\pi}{4},$ and so $\Lambda_\theta= I$. By Theorem~\ref{cone 1}, the cone defined by $L_1$ and $L_2$ is $  (I,{\bf a}) =   {\mathcal{S}},$
which proves that $L_1$ and $L_2$ are defining lines for ${\mathcal{S}}$.  
 \end{proof}

\begin{remark}
Theorem~\ref{pos def}(1) can be restated as asserting that if $\c S$ is an HR-cone that is not a unit cone, then the pair of  defining lines for $\c S$ is the projection to the $xy$-plane of the intersection of $\c S$ and a unit cone having its apex in the same location as the apex of $\c S$.
 Alternatively, the defining lines are the projection to the $xy$-plane of   the only two lines on the cone passing through the apex at an  angle of $45$ degrees from the $xy$-plane. By contrast, for a unit cone every line passing through the apex is at an  angle of $45$ degrees from the $xy$-plane, a fact reflected in statement (2) of  Theorem~\ref{pos def}. 
 \end{remark} 
 
 It remains to describe the surfaces defined by  a pair of parallel lines. Applying the relevant definitions, we have
 
 \begin{proposition} \label{cone 2}
 If $L_1$ and $L_2$ are parallel lines in the plane, then  the surface $\c S$   defined by these lines is  ${\c S}=\{(x,y,z):(x,y) \in L, |z| \geq d\},$ where 
 $L$ is the line that is equidistant from $L_1$ and $L_2$ and  $2d$ is the distance between $L_1$ and $L_2$. 
 \end{proposition} 
 
 Thus the surface $\c S$ in ${\mathbb{R}}^3$ defined by a pair of lines 
is either an HR-cone or a vertical plane with a missing midsection. 
In Theorem~\ref{cone 1} we saw that an HR-cone that is not a unit cone has a unique pair of defining lines. The same is true when $\c S$ is not an HR-cone.

 \begin{corollary} \label{pos def 2}  
 If the surface defined by a pair of lines is planar, then ${\c S}$ has a unique pair of defining lines. Let $d$ be the  distance between the half planes that comprise ${\c S}$. The defining lines  
 are  the parallel lines in the $xy$-plane that are the distance $\frac{d}{2}$ from the  projection of ${\c S}$ to the $xy$-plane.
 \end{corollary}
 
 \begin{proof}  If $L_1,L_2$ is a pair of lines that define $\c S$, then $L_1$ and $L_2$ are parallel by Theorem~\ref{cone 1}. With $2f$ the distance between $L_1$ and $L_2$,  Proposition~\ref{cone 2} implies that $f = d$ and $L_1$ and $L_2$ are a distance $\frac{d}{2}$ from the projection of $\c S$ to the $xy$-plane. 
 \end{proof}


 

   

\section{Rectangle loci as cone intersections}

We associate a  locus of points to  two pairs of lines in the plane. The locus consists of  the centers of the rectangles that are inscribed in the lines and whose diagonals reflect the pairing of the lines. As we prove, the locus can be the empty set, the entire plane, a point, a line, a line with a segment missing, or a hyperbola, with the last case being the most interesting and ubiquitous. 

By ``a pair of lines'' we continue to mean a pair of a distinct lines, and by ``two pairs of lines'' we mean two distinct pairs of lines.  While the lines in each pair are distinct, the two pairs may share a line.

\begin{definition}
The {\it rectangle locus} for  two pairs of lines is the locus of points that occur as the center of a rectangle (possibly degenerate) inscribed in the two pairs of lines and that one diagonal of the rectangle joins the lines in one pair and the other joins the lines in the other pair.
\end{definition} 

 
  Given the rectangle locus for two pairs of lines, the rectangles inscribed in  the pairing can be reconstructed from their centers, and hence from the points on the locus.
  If at least one pair consists of  parallel lines, then a point on the rectangle locus will occur as the center of more than one rectangle inscribed in the pairs, but it is not hard to work out  what these rectangles are. 
Finding the centers of all rectangles  inscribed in  four distinct lines amounts to finding  21 rectangle loci. (We are not counting the degenerate rectangles whose vertices lie on a single line.) This is because there are 6 pairs of lines among the four. For each pair that is not parallel, we can find rectangles inscribed on the two lines in the pair and having center at the intersection of the pair.  
 Next, choosing all 12 groups of two pairs that have a line in common, we apply the results of this section to describe the rectangle loci for rectangles having vertices on exactly 3 of the 4 lines.\footnote{In the case of two pairs sharing a line, it is not too difficult to see that the rectangle locus will be a degenerate hyperbola if neither pair consists of parallel lines; see \cite{OW2}.}  
  Finally, choosing the 3 groups of pairs that have no lines in common, we find the rectangle loci for rectangles having vertices on all 4 lines. 
     It is this last case that is the most substantial and the case to which we devote much of this section and the next.



  \begin{lemma} \label{bridge} 
  The rectangle locus for  two pairs of lines in the plane is the projection  of the intersection of the  surfaces defined by  the pairs.   
  \end{lemma}
  
  \begin{proof}  The proof is a matter of applying the relevant definitions.  Let $A,C$ and $B,D$ be two pairs of lines in the plane. 
    A point $p=(x,y)$ in the plane is on the rectangle locus for these two pairs   if and only if the length $2z_1$ of the line segment joining $A$ and $C$ and 
 having $p$ as its midpoint is equal to the length $2z_2$ of     
    the line segment joining $B$ and $D$ and having $p$ as its midpoint. This is the case if and only if $(x,y,{z_1})=(x,y,{z_2})$, which holds 
       if and only if $(x,y)$ is on the projection to the $xy$-plane of the intersection of the surfaces defined by the two pairs $A,C$ and $B,D$.  
  \end{proof}

  With Lemma~\ref{bridge} we can revisit the idea discussed after Definition~4.1. Given four lines, none of which are parallel, there are 15 ways to group these lines into two pairs, allowing that the pairs may share a single line in common.  Each pair of lines defines a different HR-cone. Finding the rectangle locus for two pairs then amounts to finding the projection to the $xy$-plane of the intersection of the two HR-cones defined by these pairs. A series of graphs illustrating this  can be found in \cite{Zia}.

  
  


\begin{proposition} \label{cases}
Consider two pairs of lines in the plane. 
\begin{enumerate} 
\item If one pair consists of parallel lines, then the rectangle locus for the pairs is a line, a line  
 missing an open segment, a point or the empty set. 

\item 
The rectangle locus is the entire plane if and only if each pair consists of orthogonal lines and all four lines meet at the same point.  

\item Suppose neither pair consists of parallel lines. If the crossing points of the pairs are different and  either both pairs are orthogonal or
 one pair is  a translation of the other, then 
the rectangle locus is a line.

\end{enumerate}

 \end{proposition} 
 
 \begin{proof} (1) Suppose one pair consists of parallel lines. By Proposition~\ref{cone 2}, the surface defined by this pair of  lines is a  vertical plane with missing midsection whose projection to the $xy$-plane is a line. By Lemma~\ref{bridge}, the rectangle locus for the two pairs lies on this line.  If the other pair of lines consists of parallel lines, then Lemma~\ref{bridge} implies that the rectangle locus is a line, a point or the empty set. On the other hand, if the lines in the  other pair   are not parallel, then the surface defined by this pair is  an HR-cone by Theorem~\ref{cone 1}. Since  $\c S$ is a plane with a missing midsection, the rectangle locus   is either a line or a line missing an open segment. 
 
 (2)  If the rectangle locus is the entire plane, then by (1) neither pair consists of parallel lines, and so by Theorem~\ref{cone 1} each pair defines an HR-cone. From Lemma~\ref{bridge} it follows that the apexes of the two cones coincide, and so the projection to the $xy$-plane is the entire plane only if the two cones are identical. By Theorem~\ref{pos def}, two different pairs of lines define the same cone only if the lines in each pair are orthogonal. Thus   each pair consists of  orthogonal lines and all four lines pass through the same point.  The converse is a similar application of Theorem~\ref{pos def} and Lemma~\ref{bridge}.
 
 (3)  If both  pairs consist of orthogonal lines, then Theorem~\ref{pos def}(2) and Lemma~\ref{bridge} imply that we can assume one pair is a translation of the other. Thus we need only prove (3) in the case in which one pair is a translation of the other.   By Theorem~\ref{pos def}(1),  the cone matrices $A$ and $B$ for the two pairs are the same. Let ${\bf a}$ be the  point where the lines in  the first pair meet and ${\bf b}$ the  point where the lines in the  other pair meet. 
 By Lemma~\ref{bridge},  the rectangle locus is the set of points ${\bf x}=(x,y)$  such that $({\bf x}-{\bf a})^TA({\bf x}-{\bf a})-({\bf x}-{\bf b})^TA({\bf x}-{\bf b})=0.$ The coefficients for the monomials in $x,y$ of degree $2$ are $0$, so the rectangle locus is a line.    \end{proof}

 The converse of Proposition~\ref{cases}(3) is proved in Corollary~\ref{line cor}. 
The rectangle loci in the proposition occur in rather special circumstances. It remains to describe the rectangle loci for two generically chosen pairs of lines, where neither pair consists of parallel lines, at most one pair is orthogonal, and the pairs are not translations of each other. We show in Theorem~\ref{indiff} that these are precisely the sets of pairs whose rectangle locus is a hyperbola. 

   \begin{lemma} \label{eq lemma} 
   Let ${\c P}_1$ and ${\c P}_2$ be two pairs of    lines in the plane that intersect at ${\bf a}$ and  ${\bf b}$, respectively, and let $A$ and $B$ be the cone matrices for the HR-cones defined by the two pairs. 
Let $C = A-B$ and   ${\bf c} = A{\bf a}- B{\bf b}$. If $A \ne B$, then the rectangle locus for 
the pairs ${\c P}_1$ and ${\c P}_2$
is the set of all ${\bf x} \in {\mathbb{R}}^2$ such that 
$$({\bf x} - C^{-1}{\bf c})^T
 C({\bf x} - C^{-1}{\bf c})= {\bf c}^TC{\bf c}-{\bf a}^TA{\bf a}+{\bf b}^TB{\bf b}.$$
%

\end{lemma}

\begin{proof} 
By Theorem~\ref{pos def} and Lemma~\ref{bridge},   the rectangle locus for the two pairs of lines is  the set of all points ${\bf x}$ in ${\mathbb{R}}^2$ such that $$
({\bf x}-{\bf a})^T
 A({\bf x}-{\bf a})
-
({\bf x}-{\bf b})^T
 B({\bf x}-{\bf b})=0.$$ 
By Lemma~\ref{invertible}, 
$C$ is invertible. 
Mimicking a standard calculation for summing quadratic forms, we can expand the left-hand side of this equation 
 and  
 $$({\bf x}-C^{-1}{\bf c})^T   C({\bf x}-C^{-1}{\bf c})- {\bf c}^TC^{-1}{\bf c}+{\bf a}^TA{\bf a}-{\bf b}^TB{\bf b}$$
to see that they are equal, and so the lemma follows. 
\end{proof}


 
 
  
\begin{theorem}  \label{indiff} 
The rectangle locus for two pairs of lines is a hyperbola  if and only if
neither pair consists of parallel lines, at most one pair is orthogonal, and 
 the pairs are not translations of each other. \end{theorem}

\begin{proof} 
If  the rectangle locus is a hyperbola, then  Proposition~\ref{cone 2} and Lemma~\ref{bridge} imply that neither pair consists of parallel lines, while Proposition~\ref{cases}(3) implies that the pairs are not  translations of each other and at least one pair is not orthogonal. 
Conversely, suppose  that neither pair consists of parallel lines, at least one pair is not orthogonal and 
 the pairs are not translations of each other. By Theorem~\ref{cone 1}, the surfaces defined by these two pairs are HR-cones, say $(A,{\bf a})$ and $(B,{\bf b})$. 
 Since one pair is not a translation of the other,  Theorem~\ref{pos def} implies that $A  \ne B$. 
 Lemma~\ref{eq lemma} implies the rectangle locus for the two pairs   is  the set of all ${\bf x}$ such that  $
({\bf x}-C^{-1}{\bf c})^T C({\bf x}-C^{-1}{\bf c})
= {\bf c}^TC^{-1}{\bf c}-{\bf a}^TA{\bf a}+{\bf b}^TB{\bf b},
$
where $C = A-B$ and  ${\bf c} =A{\bf a} - B{\bf b}$.  Since $C$ is a symmetric matrix and $\det(C)<0$ by Lemma~\ref{invertible}, the rectangle locus is a hyperbola (see for example \cite[p.~161]{Kendig}). 
\end{proof}

{For conditions under which the hyperbola is degenerate, see \cite{OW2} and \cite{Sch}.} 

  \begin{corollary} \label{line cor}  Given two pairs of lines in the plane, neither of which consists of parallel lines, the rectangle locus  is a line if and only if one pair is a translation of the other or both of the pairs are orthogonal. 
 \end{corollary} 
 
 \begin{proof} If the rectangle locus is a line, then Theorem~\ref{indiff} implies that one pair is a translation of the other or both pairs are orthogonal. 
 The converse follows from Proposition~\ref{cases}(3). 
  \end{proof}

  \begin{corollary} The rectangle locus for a pair of lines in the plane is either the empty set, a point, a line with an open segment missing, a line, a hyperbola or the entire plane. 
 \end{corollary}

\begin{proof} Apply 
 Proposition~\ref{cases}, Theorem~\ref{indiff} and Corollary~\ref{line cor}.
 \end{proof}

 

\begin{corollary} \label{late cor} The intersection of two HR-cones projects to a hyperbola in the $xy$-plane.
\end{corollary}

 

 {Figure 1 illustrates Corollary~\ref{late cor}.}
 Using the determinant criterion for classifying conics, it is not hard to give an example
of  an HR-cone and a cone as in Equation (1) whose intersection projects  to an ellipse or a parabola in the plane. In light of Corollary~\ref{late cor}, the second cone cannot be an HR-cone.

 \begin{figure}[h] \label{picture}
 \begin{center}
 \includegraphics[width=0.65\textwidth,scale=1.2]{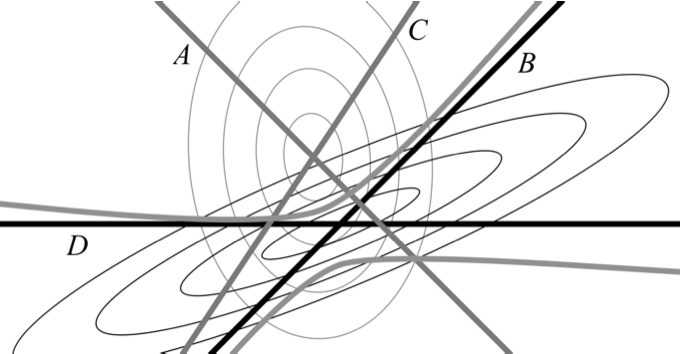} 
 \end{center}
 \caption{
{The level curves of the two HR-cones defined by the pairs $A,C$ and $B,D$.  The rectangle locus, which is the projection of the intersection of the two cones to the plane, {is the hyperbola that passes through the intersections of the level curves}.} 
  }
\end{figure}

 

\begin{remark}
Given  two pairs of lines whose rectangle locus is a hyperbola, let   $(A,{\bf a})$ and $(B,{\bf b})$ be the HR-cones defined by the two pairs. If ${\bf p}$ is the center of the hyperbola, then Lemma~\ref{eq lemma} and the fact that by Theorem~\ref{indiff} neither pair is a translation of the other
 imply that $A \ne B$ and the pair of asymptotes for the rectangle locus is the degenerate hyperbola $({\bf x} - {\bf p})^T(A-B)({\bf x}-{\bf p})=0.$
 Lemma~\ref{bridge} implies this  
  pair of lines is the projection   of the intersection of the  two HR-cones $(A,{\bf p})$ and $(B,{\bf p})$.  Thus a translated copy of the asymptotes of the rectangle locus can be found by moving the apex of one cone to the apex of the other and projecting the intersection to the $xy$-plane. 
 \end{remark}

\begin{remark} 
 Remark~\ref{axes rem}, Lemma~\ref{bridge} and Corollary~\ref{line cor}  imply that if ${\c A}$ and ${\c B}$ are HR-cones, then the intersection of these two cones lies in a vertical plane if and only if either both cones are unit cones or the major axes of the level curves of the two cones are parallel. 
Suppose the intersection of the two cones does not lie in a 
 single vertical plane. Then the intersection  is a curve in ${\mathbb{R}}^3$ consisting of four branches,  with each branch having another branch among the four as a reflection through the $xy$-axis. By Lemma~\ref{bridge} and Theorem~\ref{indiff}, each reflected pair projects to a branch of a hyperbola in the $xy$-plane. However, a branch of the intersection of the two cones need not  lie in a plane.  For example, consider the two HR-cones \begin{center} $z^2 = (x-1)^2+y^2$ \: and \: $ z^2 = 2x^2+\frac{1}{2}y^2.$
\end{center}
The four points $(1,2,4),(2,\sqrt{14},15),(3,\sqrt{28},32),(4,\sqrt{46},55)$ lie on a single branch but do not lie in a plane since $$\det\left(\begin{bmatrix}
1 & 2 & 4 & 1 \\
2 & \sqrt{14} & 15 & 1 \\
3 & \sqrt{28} & 32 & 1 \\
4 & \sqrt{46} & 55 & 1 \\
\end{bmatrix}\right) = 18\sqrt{14}-36\sqrt{7}+6\sqrt{46}
-12\ne 0.$$ 
\end{remark}

\section{Line pairings that share the same rectangle locus}

The results of the last section suggest two questions: {\it Is there anything special about the hyperbolas that occur as rectangle loci for two pairs of lines?} and {\it If the rectangle locus for two pairs is a hyperbola, how do the shape and location of the hyperbola reflect properties of the two pairs that produced it?}
 The answer to the first question, as we show in  Theorem~\ref{realize}, is  ``no.''  
For the second question, we give in this same theorem  a first step towards an answer. 
 


We mention first an easy but striking consequence of the interpretation in Lem\-ma~\ref{bridge} of rectangle loci in terms of cones.

\begin{proposition} \label{unique} 
Let ${\c P}$ and ${\c Q}$ be two pairs of  lines in the plane. If the lines in ${\c P}$ are orthogonal, then   any rotation of ${\c P}$  will not change the rectangular locus of ${\c P}$ and ${\c Q}$.  
\end{proposition}



\begin{proof}  
This is a consequence of the fact that rotating a circular cone does not change the cone. In particular,  Theorem~\ref{pos def} implies that all the pairs of orthogonal lines having the same crossing point  define the same unit cone so   Lemma~\ref{bridge} assures that the same rectangle locus is obtained. 
\end{proof}

\begin{example}  Uniqueness can also fail when neither pair is orthogonal and both pairs are anchored at the same point. 
 For example, consider the following positive definite matrices, each of which has determinant~$1$.
$$
A_1=\begin{bmatrix} 
\frac{5}{8}& \frac{1}{2} \\ \frac{1}{2} & 2 \\
\end{bmatrix}, 
B_1=\begin{bmatrix} 
2& 1 \\ 1 & 1 \\
\end{bmatrix},
A_2=\begin{bmatrix} 
1& 0 \\ 0 & 1 \\
\end{bmatrix}, 
B_2=\begin{bmatrix} 
\frac{13}{8}& \frac{3}{2} \\ \frac{3}{2} & 2 \\
\end{bmatrix}.
$$
Then $A_1-B_2 = A_2-B_1$. By Proposition~\ref{ugh}, the HR-cones $(A_1,{\bf 0})$, $(B_2,{\bf 0})$ determine  two different pairs of defining lines than the HR-cones $ (A_2,{\bf 0})$, $(B_1,{\bf 0})$. Yet since $A_1-B_2 = A_2-B_1$, Lemma~\ref{eq lemma} implies that the  two pairs of defining lines for $(A_1,{\bf 0})$ and $(B_2,{\bf 0})$
determine the same rectangle locus as the two pairs of defining lines for $ (A_2,{\bf 0})$ and $(B_1,{\bf 0})$. 
\end{example}


 

We show next that every hyperbola occurs as the rectangular locus for two pairs of lines. We do this  by giving a method for finding  all pairs of HR-cones whose intersection projects to the hyperbola. See \cite{Zia} for an animation illustrating the pencil-like behavior of different pairs of HR-cones whose intersection projects to the same hyperbola.  

\begin{theorem} \label{realize} Each  hyperbola ${\mathcal{H}}$ in the plane  is the projection of the intersection of  two HR-cones, and the set of all such   pairs whose intersection projects to ${\mathcal{H}}$ is faithfully parameterized by a semialgebraic surface in ${\mathbb{R}}^4$.  
%
\end{theorem}

\begin{proof} 
Let ${\mathcal{H}}$ be a hyperbola in the plane. After translation we may assume that ${\mathcal{H}}$   has its center at the origin and so 
there is a symmetric   $2 \times 2$ matrix $C$ with $\det(C)<0$ such that  ${\mathcal{H}}$ is given by the equation ${\bf x}^TC{\bf x}=1$.  By the spectral theorem, after a suitable rotation we may also assume that  there are real numbers  $\lambda_1,\lambda_2$ such that $$C = 
\begin{bmatrix} 
\lambda_1 & 0 \\
0 & \lambda_2 \\ 
\end{bmatrix}.$$ The fact that $\det(C) < 0$ implies that $\lambda_1$ and $\lambda_2$ are  nonzero and have different signs.

We  first describe all the  positive definite symmetric $2 \times 2$ matrices $A,B$ with determinant $1$ such that $A-B=C$.  As a way to find entries for these matrices, consider the quadratic polynomial in the variables $u,v$ given by  $$f(u,v) = v^2  
+\left(\frac{\lambda_2}{\lambda_1}\right)u^2+\lambda_2 u+1.$$ This polynomial   defines a hyperbola since  $\lambda_1$ and $\lambda_2$ having different signs implies the discriminant of the conic is positive. Moreover, since this hyperbola is symmetric about the $x$-axis, we may choose a point $(c,b)$ on the hyperbola such that $c>\max\{0,-\lambda_1\}$.
Let
  $$A = \begin{bmatrix} c+\lambda_1 & b \\ b & \frac{1+b^2}{c+\lambda_1} \\ \end{bmatrix}, \: B=\begin{bmatrix} c & b \\ b & \frac{1+b^2}{c} \\ \end{bmatrix}.$$
 Then  $A$ and $B$ are positive definite symmetric matrices with determinant $1$. A calculation shows that  $A-B = C$ since   $f(c,b)=0$.

Any pair $A,B$  of positive definite symmetric matrices with determinant $1$ and $A-B = C$ must arise this way and, in particular, $f(c,b)$ must be $0$ where $c,b$ are the entries on the first row of $B$. This observation will be implicitly used again at the end of the proof where it translates into the fact that a semialgebraic surface  parameterizes all pairings of cones that produce the hyperbola $\c H$.

To prove now that ${\c H}$ is the projection of the intersection of HR cones, it is enough by 
  Lemma~\ref{eq lemma}  to show   there are vectors ${\bf a},{\bf b}$ such that $0 = A{\bf a}- B{\bf b}$ and ${\mathcal{H}}$ is the set of points ${\bf x}$ with 
${\bf x}^T
 C{\bf x} = -{\bf a}^TA{\bf a}+{\bf b}^TB{\bf b}.$
%
 %
(The fact that the hyperbola is centered at the origin accounts for the absence of the terms involving  ``${\bf c}$'' in Lemma~\ref{eq lemma}.)
Therefore, since $\c H$ is defined by the equation ${\bf x}^T
 C{\bf x}=1$, 
we must show  
there are vectors ${\bf a}$ and  ${\bf b}$ such that \begin{center} $ 1 =- {\bf a}^TA{\bf a}+{\bf b}^TB{\bf b}$ \: and \: $0 = A{\bf a} - B{\bf b}$.   \end{center}
To obtain $0 = A{\bf a} -B {\bf b}$ requires
 ${\bf a} = A^{-1}B{\bf b}$.  
For each vector ${\bf b}$, setting ${\bf a} = A^{-1}B{\bf b}$ yields
\begin{eqnarray*} 
 -{\bf a}^TA{\bf a}+{\bf b}^TB{\bf b} & = & -{\bf b}^TB A^{-1} AA^{-1}B{\bf b} + {\bf b}^TB{\bf b} \\
 &= & {\bf b}^T(-BA^{-1}B+B){\bf b} \: = \: {\bf b}^T((-B+A)A^{-1}B){\bf b} \\
 & = & {\bf b}^TCA^{-1}B{\bf b}.
 \end{eqnarray*}
 Thus it remains to show there is a vector ${\bf b}$ such that $1 =  {\bf b}^TCA^{-1}B{\bf b}.$ Since $\det(C) <0$ and $\det(A) = \det(B) = 1$, we have that $\det(CA^{-1}B) < 0$.  Also, $CA^{-1}B = BA^{-1}C$, so $CA^{-1}B$  is symmetric. 
 Thus the locus of points ${\bf x}$ with the property that 
  ${\bf x}^T
 CA^{-1}B{\bf x}= 1$
 is a hyperbola. Any point on this hyperbola  will do for ${\bf b}$. 

Finally, we claim that the pairs of cones that project to the given hyperbola are parameterized by a semialgebraic   surface in ${\mathbb{R}}^4$. We have already shown that each such pair is determined by a point $(u,v,x,y)$ such that $(u,v)$ lies on the hyperbola   in ${\mathbb{R}}^2$ given by
\begin{eqnarray}\label{new 1}  v^2  
+\left(\frac{\lambda_2}{\lambda_1}\right)u^2+\lambda_2 u  = -1,
\end{eqnarray}
 while ${\bf x}=(x,y)$ lies on the hyperbola in ${\mathbb{R}}^2$  defined by  
$
 {\bf x}^T
 CA^{-1}B{\bf x} = 1$. Unpacking the latter equation results in  
  \begin{eqnarray} \label{new 2}
  (\lambda_1\lambda_2 u +\lambda_1)x^2+2\lambda_1\lambda_2 vxy-(\lambda_2^2u+\lambda_1\lambda_2^2-\lambda_2)y^2 = 1.
  \end{eqnarray}
 The pairs of cones that project to the given hyperbola  are determined by the points $(x,y,u,v)$ that lie on the surface in ${\mathbb{R}}^4$ obtained by intersecting  the two hypersurfaces~(\ref{new 1}) and~(\ref{new 2}) subject to the constraint $u>\max\{0,-\lambda_1\}$. Specifically, $(x,y)$ is the point ${\bf b}$ and $(u,v)$ is the first row of the matrix $B$, which in turn determines the second row of $B$.  With $A = C-B$ and ${\bf a} = A^{-1}B{\bf b}$,  the intersection of the  pair of HR-cones  $(A,{\bf a})$ and $(B,{\bf b})$ projects to the given hyperbola. That this parameterization is faithful follows from Proposition~\ref{ugh}. 
\end{proof}


\begin{corollary} 
Let ${\mathcal{H}}$ be a hyperbola in the plane. The set of all pairs of pairs of lines that have $\c H$ as their rectangle locus is 
 faithfully parameterized by a semialgebraic surface in ${\mathbb{R}}^4$.  \end{corollary}

\begin{proof} Apply Lemma~\ref{bridge} and Theorem~\ref{realize}.
\end{proof}

  \section{Additional context}
  
  There is an extensive literature on  rectangles and squares inscribed in Jordan  curves; see for example \cite{Mat} and the discussion in \cite[p.~1]{Sch}. That it is always possible to 
 find a rectangle inscribed  in  a Jordan curve is due to Vaughan (see \cite[p.~71]{Mey}). Meanwhile, 
 the problem of finding inscribed squares  remains open in  full generality, although many important cases have been resolved (see for example \cite{Mat}). The existence of inscribed rectangles for a polygon is a simpler matter. 
 While we have focused on rectangles inscribed in four lines rather than 
 in polygons, it is not hard to apply the ideas discussed after Definition 4.1 to obtain a description of the centers of the rectangles inscribed  in polygons. (This involves ruling out the ``exscribed'' rectangles that our methods find when viewing the polygon as having sides lying  on lines. It is straightforward to do this.)
  As the discussion after Definition 4.1 indicates, there will be 21 rectangle loci involved in the description of the rectangles inscribed in a quadrilateral.  
 
   
   \medskip

{\it Acknowledgement.} We thank the referee for helpful comments that improved the presentation.

\end{document}